\documentclass[11pt]{article}
\usepackage{amsmath,amsfonts,amssymb,amsthm,enumerate,graphicx}
\usepackage{tikz,authblk}
\usepackage{subfig}
\usepackage{hyperref}

\newtheorem{theorem}{{\bf Theorem}}[section]
\newtheorem{conjecture}[theorem]{{\bf Conjecture}}
\newtheorem{corollary}[theorem]{{\bf Corollary}}

\newtheorem{question}[theorem]{{\bf Question}}

\newcommand{\FF}{ \ensuremath{\mathbb{F}}}

\newcommand{\ZZ}{ \ensuremath{\mathbb{Z}}}

\newcommand{\TPSS}{S^{\hspace{.2mm}2}\! \times \hspace{-3.3mm}_{-} \,
S^{\hspace{.1mm}1}}

\begin{document}

\author[1] {Bhaskar Bagchi}
\author[2] {Basudeb Datta}
\author[3] {Jonathan Spreer}

 \affil[1] {Theoretical Statistics and Mathematics Unit, Indian Statistical Institute, Bangalore 560\,059,
 India. bbagchi@isibang.ac.in}
 \affil[2] {Department of Mathematics, Indian Institute of Science, Bangalore 560\,012, India.
 dattab@math.iisc.ernet.in.}
 \affil[3] {School of Mathematics and Physics, The University of Queensland, Brisbane QLD 4072, Australia.
 j.spreer@uq.edu.au.}

\title{A characterization of tightly triangulated 3-manifolds}


\date{January 08, 2016}

\maketitle

\vspace{-10mm}

\begin{abstract}
For a field $\mathbb{F}$, the notion of $\mathbb{F}$-tightness of simplicial complexes was introduced by K\"{u}hnel. K\"{u}hnel
and Lutz conjectured that any $\mathbb{F}$-tight triangulation of a closed manifold is the most economic of all possible
triangulations of the manifold. The boundary of a triangle is the only 
$\mathbb{F}$-tight triangulation of a closed
1-manifold. A triangulation of a closed 2-manifold is $\mathbb{F}$-tight if and only if it is $\mathbb{F}$-orientable and
neighbourly. In this paper we prove that a triangulation of a closed 3-manifold is 
$\mathbb{F}$-tight if and only if it
is $\mathbb{F}$-orientable, neighbourly and stacked. In consequence, the 
K\"{u}hnel-Lutz conjecture is valid in
dimension $\leq 3$.
\end{abstract}

\noindent {\small {\em MSC 2010\,:} 57Q15, 57R05.

\noindent {\em Keywords:} Stacked spheres; Stacked manifolds; Triangulations of 3-manifolds; Tight
triangulations.}

\section{Introduction}

All simplicial complexes considered in this paper are finite and abstract. The
vertex set of a simplicial complex $X$ will be denoted by $V(X)$. For $A \subseteq V(X)$, the induced subcomplex $X[A]$ of $X$ on the vertex set $A$ is defined by
$X[A] := \{\alpha\in X \, : \, \alpha\subseteq A\}$. 
For $x\in V(X)$, the subcomplexes $\{\alpha\in X \, : \, x \not\in\alpha\} = X[V(X)\setminus\{x\}]$ and $\{\alpha\in X : x\not\in\alpha, \alpha\sqcup\{x\}\in X\}$ are called the {\em antistar} and the {\em link} of $x$ in $X$, respectively. 
A simplicial complex $X$ is said to be a {\em triangulated $($closed$)$ manifold} if it triangulates a (closed) manifold, i.e., if the geometric carrier $|X|$ of $X$ is 
a (closed) topological manifold. A  triangulated closed $d$-manifold $X$ is said to be {\em $\FF$-orientable} if $H_d(X; \FF) \neq 0$. 
If two triangulated $d$-manifolds $X$ and $Y$ intersect precisely in a common $d$-face $\alpha$ then $X\# Y := (X \cup Y) \setminus \{\alpha\}$ triangulates the connected sum $|X|\# |Y|$ and is called the {\em connected sum of $X$ and $Y$ along 
$\alpha$}.

For our purpose, a {\em graph} may be defined as a simplicial complex of dimension $\leq 1$. For $n\geq 3$, the
{\em $n$-cycle} $C_n$ is the unique $n$-vertex connected graph in which each vertex lies on exactly two edges. For $n\geq
1$, the {\em complete graph} $K_n$ is the $n$-vertex graph in which any two vertices form an edge. For $m, n\geq
1$, the {\em complete bipartite graph} $K_{m, n}$ is the graph with $m+n$ vertices and $mn$ edges in which each
of the first $m$ vertices forms an edge with each of the last $n$ vertices. Two graphs are said to be {\em
homeomorphic} if their geometric carriers are homeomorphic. A graph is said to be {\em planar} if it is a
subcomplex of a triangulation of the 2-sphere $S^2$. In this paper, we shall have an occasion to use the easy
half of Kuratowski's famous characterization of planar graphs \cite{Book_BM}: A graph is planar if and only if it
has no homeomorph of $K_5$ or $K_{3, 3}$ as a subgraph.

If $\FF$ is a field and $X$ is a simplicial complex then, following K\"{u}hnel \cite{Ku}, we say that $X$ is
{$\FF$-tight} if (a) $X$ is connected, and (b) the $\FF$-linear map $H_{\ast}(Y; \FF) \to H_{\ast}(X; \FF)$,
induced by the inclusion map $Y \hookrightarrow X$, is injective for every induced subcomplex $Y$ of $X$.

If $X$ is a simplicial complex of dimension $d$, then its {\em face vector} $(f_0, \dots, f_d)$ is  defined by
$f_i=f_i(X) := \#\{\alpha\in X \, : \, \dim(\alpha)=i\}, \,  0\leq i\leq d$. A simplicial complex $X$ is said to be
{\em neighbourly} if any two of its vertices form an edge, i.e., if $f_{1}(X) = \binom{f_0(X)}{2}$.

A simplicial complex $X$ is said to be {\em strongly minimal} if, for every triangulation $Y$ of the geometric
carrier $|X|$ of $X$, we have $f_i(X) \leq f_i(Y)$ for all $i$, $0\leq i\leq \dim(X)$. Our interest in the notion
of $\FF$-tightness mainly stems from the following famous conjecture \cite{KL}.

\begin{conjecture}[K\"{u}hnel-Lutz] \label{conj:KL}
For any field $\FF$, every $\FF$-tight triangulated closed manifold is strongly minimal.
\end{conjecture}

Following Walkup \cite{Wa} and McMullen-Walkup \cite{MW}, a triangulated ball $B$ is said to be {\em stacked} if
all the faces of $B$ of codimension 2 are contained in the boundary $\partial B$ of $B$. A triangulated sphere
$S$ is said to be {\em stacked} if there is a stacked ball $B$ such that $S= \partial B$. This notion was extended to
triangulated manifolds by Murai and Nevo \cite{MNStacked}. Thus, a triangulated manifold $\Delta$ with boundary
is said to be {\em stacked} if all its faces of codimension 2 are contained in the boundary $\partial\Delta$ of
$\Delta$. A triangulated closed manifold $M$ is said to be {\em stacked} if there is a stacked triangulated
manifold $\Delta$ such that $M = \partial \Delta$. A triangulated manifold is said to be {\em locally stacked}
if all its vertex links are stacked spheres or stacked balls.
The main result of this paper is the following characterization of $\FF$-tight  triangulated closed 3-manifolds,
for all fields $\FF$.

\begin{theorem} \label{theo:main}
A triangulated closed $3$-manifold $M$ is $\FF$-tight if and only if $M$ is $\FF$-orientable, neighbourly and
stacked.
\end{theorem}

The special case of Theorem \ref{theo:main}, where ${\rm char}(\FF)\neq 2$, was proved in our previous paper
\cite{BDS1}. In this paper we conjectured \cite[Conjecture 1.12]{BDS1} the validity of Theorem
\ref{theo:main} in general.

As a consequence of Theorem \ref{theo:main}, we show that the K\"{u}hnel-Lutz conjecture (Conjecture \ref{conj:KL})
is valid up to dimension 3. Thus,

\begin{corollary} \label{coro:KL}
If $M$ is an $\FF$-tight triangulated closed manifold of dimension $\leq 3$, then $M$ is strongly minimal.
\end{corollary}

As a second consequence of Theorem \ref{theo:main}, we show:

\begin{corollary} \label{coro:S2S1}
The only closed topological $3$-manifolds which may possibly have $\FF$-tight triangulations are $S^3$,
$(\TPSS)^{\# k}$ and $(S^2\times S^1)^{\# k}$, where $k$ is a positive integer such that $80k+1$ is a perfect
square.
\end{corollary} 

K\"uhnel conjectured that any 
triangulated closed 3-manifold $M$ satisfies $(f_0 (M)-4)\times$ $(f_0(M)-5)\geq 20 
\beta_1(M; \FF)$. (This is a part of his Pascal-like triangle of conjectures 
reported in \cite{Lu}.) This bound was proved by Novic and Swartz in \cite{NS}. Burton et al proved in \cite{BDSS1} that if the equality holds in 
this inequality then $M$ is neighbourly and locally stacked. (Actually, these authors stated this result for $\FF=\ZZ_2$, but their argument goes through for 
all fields $\FF$.) In \cite{BaMu}, the first author proved that the equality holds in 
this inequality if and only if $M$ is neighbourly and stacked. In \cite{Mu}, Murai generalized this to all dimensions $\geq 3$. Another consequence of Theorem \ref{theo:main} is:

\begin{corollary} \label{coro:Ku}
A triangulated closed $3$-manifold $M$ is $\FF$-tight if and only 
if $M$ is $\FF$-orientable and $(f_0(M)-4)(f_0(M)-5) = 20\beta_1(M; \FF)$.
\end{corollary}

In \cite{BDSS2}, $\ZZ_2$-tight triangulations of $(\TPSS)^{\# k}$ were constructed for $k = 1, 30, 99, 208, 357$ and 546. However, we do not know any $\FF$-tight triangulations of $(S^2\times S^1)^{\#
k}$.

\begin{question} \label{question:S2S1}
Is there any positive integer $k$ for which $(S^2\times S^1)^{\# k}$ has an $\FF$-tight triangulation?
\end{question}

\section{Proofs}

The following result is Theorem 3.5 of \cite{BDS1}.

\begin{theorem} \label{theo:BDS1_T3.5}
Let $C$ be an induced cycle in the link $S$ of a vertex $x$ in an $\FF$-tight simplicial complex $X$. Then the induced subcomplex of $X$ on the vertex set of the cone $x\ast C$ is a neighbourly triangulated closed $2$-manifold.
\end{theorem}

If, in Theorem \ref{theo:BDS1_T3.5}, $C$ is an $n$-cycle then the triangulated 2-manifold guaranteed by this theorem has $n+1$ vertices, $n(n+1)/2$ edges and hence $n(n+1)/3$ triangles. Thus 3 divides $n(n+1)$, i.e., $n \not\equiv 1$ (mod 3). Therefore, Theorem \ref{theo:BDS1_T3.5} has the following immediate consequence.

\begin{corollary} \label{coro:C2.2}
Let $X$ be an $\FF$-tight simplicial complex. Let $S$ be the link of a vertex in $X$. Then $S$ has no induced $n$-cycle for $n \equiv 1$ $($mod $3)$.
\end{corollary}

We recall that the M\"{o}bius band has a unique 5-vertex triangulation $\mathcal{M}$. The boundary of $\mathcal{M}$ is a 5-cycle $C_5$. The simplicial complex $\mathcal{M}$ may be uniquely recovered from $C_5$ as follows. The triangles of $\mathcal{M}$ are $\{x\}\cup e_x$, where, for each vertex $x$ of $C_5$, $e_x$ is the edge of $C_5$ opposite to $x$. We also note the following consequence of Theorem \ref{theo:BDS1_T3.5}.

\begin{corollary} \label{coro:C2.3}
Let $S$ be the link of a vertex $x$ in an $\FF$-tight simplicial complex $X$. Let $C$ be an induced cycle in $S$.
\begin{enumerate}
\item[{\rm (a)}] If $C$ is a $3$-cycle, then it bounds a triangle of $X$. \vspace{-2.5mm}
\item[{\rm (b)}] If $C$ is a $5$-cycle then it bounds an induced subcomplex of $X$ isomorphic to the $5$-vertex M\"{o}bius band.
\end{enumerate}
\end{corollary}

\begin{proof}
If $C$ is a 3-cycle, then the induced subcomplex of $X$ on the
vertex set of $x\ast C$ is a neighbourly, 4-vertex, triangulated closed 2-manifold, which must be the boundary complex $\mathcal{T}$ of the tetrahedron. But all four possible triangles occur in $\mathcal{T}$, and $C$ bounds one of them. If $C$ is a 5-cycle then the induced subcomplex $X[V(x\ast C)]$ of $X$ is a neighbourly, 6-vertex, triangulated closed 2-manifold, which must be the unique 6-vertex triangulation $\mathbb{RP}^2_6$ of the real projective plane. Therefore, the induced subcomplex $X[V(C)]$ of $X$ is the antistar of the vertex $x$ in $\mathbb{RP}^2_6$, which is the 5-vertex M\"{o}bius band.
\end{proof}

Let $\mathcal{T}$ and $\mathcal{I}$ denote the boundary complexes of the tetrahedron and the icosahedron, respectively. Thus the faces of $\mathcal{T}$ are all the proper subsets of a set of four vertices. Up to isomorphism, the 20 triangles of $\mathcal{I}$ are as follows:
\begin{align} \label{icosahedron}
& 012, 015, 023, 034, 045, 124^{\prime}, 153^{\prime}, 13^{\prime}4^{\prime}, 235^{\prime}, 24^{\prime}5^{\prime}, 341^{\prime}, \nonumber \\
&  31^{\prime}5^{\prime}, 452^{\prime}, 41^{\prime}2^{\prime}, 52^{\prime}3^{\prime}, 
0^{\prime}1^{\prime}2^{\prime}, 0^{\prime}1^{\prime}5^{\prime}, 0^{\prime}2^{\prime}3^{\prime}, 0^{\prime}3^{\prime}4^{\prime}, 0^{\prime}4^{\prime}5^{\prime}.
\end{align}

The following is Corollary 5.5 of \cite{BDS1}.

\begin{theorem} \label{theo:BDS1_T5.6}
Let $S$ be a triangulated $2$-sphere which has no induced $n$-cycle for any
$n \equiv 1$ $($mod $3)$. Then $S$ is a connected sum of finitely many copies of $\mathcal{T}$ and $\mathcal{I}$ $($in some order$)$.
\end{theorem}

As an immediate consequence of Corollary \ref{coro:C2.2} and Theorem \ref{theo:BDS1_T5.6}, we have:

\begin{corollary} \label{coro:C2.5}
Let $S$ be the link of a vertex in an $\FF$-tight triangulated closed $3$-manifold $M$. Then $S$ is a connected sum of finitely many copies of $\mathcal{T}$ and $\mathcal{I}$ $($in some order$)$.
\end{corollary}

\begin{proof}[Proof of Theorem $\ref{theo:main}$]
Let $M$ be an $\FF$-orientable, neighbourly, stacked, triangulated closed 3-manifold. 
Then $M$ is $\FF$-tight by the case $k=1$ of Theorem 2.24 in \cite{BDStacked}. 
This proves the ``if part". Conversely, let $M$ be $\FF$-tight. Since any $\FF$-tight  triangulated closed manifold is neighbourly and $\FF$-orientable (Lemmas 2.2 and 2.5 in \cite{BDS1}), it follows that $M$ is $\FF$-orientable and neighbourly. To complete the proof of the ``only if" part, it suffices to show that $M$ must be stacked. By \cite[Theorem 2.24]{BDStacked}, every locally stacked, $\FF$-tight, triangulated closed 3-manifold is automatically stacked. So, it is enough to show that if $S$ is the link of an arbitrary vertex of $M$, then $S$ is a stacked 2-sphere. By Corollary \ref{coro:C2.5}, $S = S_1\# \cdots\# S_m$, where each $S_i$ is either $\mathcal{T}$ or $\mathcal{I}$. Since any connected sum of copies of $\mathcal{T}$ is stacked (as may be seen by an easy induction on the number of summands), it suffices to show that no $S_i$ can be $\mathcal{I}$. 

Suppose, on the contrary, that $S_i = \mathcal{I}$ for some index $i$. We may take the triangles of $\mathcal{I}$ to be as given in \eqref{icosahedron}. Note that each triangle of $S_i$ is either a triangle of $S$ or its boundary is an induced 3-cycle of $S$. Since $S\subseteq M$, Corollary \ref{coro:C2.3}\,(a) implies that each triangle of $S_i$ is a triangle of $M$. Thus, $\mathcal{I}\subseteq M$. In particular, $012$ and $023$ are triangles of $M$. Also, we have the following induced 5-cycles (among others) in $\mathcal{I}$, and hence in $S$.

\setlength{\unitlength}{2.6mm}

\begin{picture}(50,14)(0,-1)

\thicklines

\put(4,10){\line(1,0){4}} \put(2,7){\line(2,3){2}} \put(10,7){\line(-2,3){2}} \put(6,3){\line(-1,1){4}}
\put(6,3){\line(1,1){4}}

\put(3,10.5){\mbox{0}} \put(8.5,10.5){\mbox{2}} \put(1.3,5.5){\mbox{5}}
 \put(10,7.5){\mbox{$4^{\prime}$}} \put(6.5,2.2){\mbox{$3^{\prime}$}}

\put(16,10){\line(1,0){4}} \put(14,7){\line(2,3){2}} \put(22,7){\line(-2,3){2}} \put(18,3){\line(-1,1){4}} \put(18,3){\line(1,1){4}}

\put(15,10.5){\mbox{0}} \put(20.5,10.5){\mbox{1}} \put(13.3,5.5){\mbox{3}}
 \put(22,7.5){\mbox{$4^{\prime}$}} \put(18.5,2.2){\mbox{$5^{\prime}$}}




\put(28,10){\line(1,0){4}} \put(26,7){\line(2,3){2}} \put(34,7){\line(-2,3){2}} \put(30,3){\line(-1,1){4}} \put(30,3){\line(1,1){4}}

\put(27,10.5){\mbox{0}} \put(32.5,10.5){\mbox{3}} \put(25.3,5.5){\mbox{5}}
 \put(34,7.5){\mbox{$1^{\prime}$}} \put(30.5,2.2){\mbox{$2^{\prime}$}}

\put(40,10){\line(1,0){4}} \put(38,7){\line(2,3){2}} \put(46,7){\line(-2,3){2}} \put(42,3){\line(-1,1){4}} \put(42,3){\line(1,1){4}}

\put(39,10.5){\mbox{0}} \put(44.5,10.5){\mbox{1}} \put(37.3,5.5){\mbox{4}}
 \put(46,7.5){\mbox{$3^{\prime}$}} \put(42.5,2.2){\mbox{$2^{\prime}$}}

\put(2,0){\mbox{\bf Figure\,1\,: Some induced 5-cycles in \boldmath{$\mathcal{I}$}}}
\end{picture}

Hence Corollary \ref{coro:C2.3}\,(b) gives us eight more triangles of $M$ through the vertex 0, namely, $023^{\prime}$, $03^{\prime}4^{\prime}$, $015^{\prime}$, $034^{\prime}$, $04^{\prime}5^{\prime}$, $032^{\prime}$, $012^{\prime}$, 
$02^{\prime}3^{\prime}$.  Thus, if $S^{\prime}$ is the link  of the vertex 0 in $M$, then we have the graph of Fig. 2 as a subcomplex of $S^{\prime}$. 


\setlength{\unitlength}{3.55mm}
\begin{picture}(20,13)(0,0)

\thicklines

\put(3,3){\line(0,1){8}} \put(11,3){\line(-1,1){8}} \put(19,3){\line(-1,1){8}} \put(19,3){\line(-2,1){16}}
\put(19,3){\line(0,1){8}}

\put(3,3){\line(2,1){5}} \put(8.8,5.9){\line(2,1){1.7}} \put(11.4,7.2){\line(2,1){1.7}}
\put(19,11){\line(-2,-1){5}}

\put(11,3){\line(0,1){3.5}} \put(11,7.5){\line(0,1){3.5}}

\put(3,3){\line(1,1){3.7}} \put(7.15,7.15){\line(1,1){1}} \put(8.6,8.6){\line(1,1){2.4}}

\put(15.3,7.3){\line(1,1){3.7}} \put(13.85,5.85){\line(1,1){1}} \put(11,3){\line(1,1){2.4}}

\put(1.7,10.7){\mbox{1}} \put(1.7,6.5){\mbox{$5^{\prime}$}} \put(1.7,2.5){\mbox{$4^{\prime}$}}

\put(9.5,10.7){\mbox{3}} \put(19.6,10.7){\mbox{$3^{\prime}$}} \put(9.5,2.5){\mbox{2}} \put(19.6,2.5){\mbox{$2^{\prime}$}}

\put(2.7,2.7){\mbox{$\bullet$}} \put(2.7,6.7){\mbox{$\bullet$}} \put(2.7,10.7){\mbox{$\bullet$}}

\put(10.7,2.7){\mbox{$\bullet$}} \put(10.7,10.7){\mbox{$\bullet$}}

\put(18.7,2.7){\mbox{$\bullet$}} \put(18.7,10.7){\mbox{$\bullet$}}

 \put(2,0.5){\mbox{\bf Figure\,2\,: A homeomorph of \boldmath{$K_{3,3}$}}}
\end{picture}

So we have a homeomorph of $K_{3,3}$ as a subcomplex of the triangulated 2-sphere $S^{\prime}$. This is a contradiction since $K_{3,3}$ is not a planar graph. 
\end{proof}

\begin{proof}[Proof of Corollary $\ref{coro:KL}$]
Let $M$ be an $\FF$-tight triangulated closed $d$-manifold, $d\leq 3$. By Lemma 2.2 in \cite{BDS1}, $M$ is neighbourly. But the boundary complex of the triangle is the only neighbourly triangulated closed 1-manifold. This is trivially the strongly minimal triangulation of $S^1$. So, we have the result for $d=1$. Next let $d=2$. Let $N$ be another triangulation of $|M|$. Let $(f_0, f_1, f_2)$ be the face vector of $N$. Let $\chi$ be the Euler characteristic of $M$ (hence also of $N$). Then $f_0-f_1+f_2 = \chi$ and $2f_1=3f_2$. Therefore we get $f_1=3(f_0-\chi)$ and $f_2=2(f_0-\chi)$. Thus, $f_1$ and $f_2$ are strictly increasing functions of $f_0$. 
So, it is sufficient to show that $f_0 \geq f_0(M)$. 
Now, trivially, $f_1\leq \binom{f_0}{2}$, with equality if and only if $N$ is neighbourly. Substituting $f_1=3(f_0-\chi)$ in this inequality, we get that $f_0(f_0-7)\geq -6\chi= f_0(M)(f_0(M)-7)$. This implies that $f_0\geq f_0(M)$. Thus, $M$ is strongly minimal. So we have the result for $d=2$. If $d=3$ then, by Theorem \ref{theo:main}, $M$ is stacked and hence is locally stacked. But any locally stacked, $\FF$-tight triangulated closed manifold is strongly minimal by Corollary 3.13 in \cite{BDStellated}. So, we are done when $d=3$.  \end{proof}

\begin{proof}[Proof of Corollary $\ref{coro:S2S1}$]
Let $M$ be a closed 3-manifold which has an $\FF$-tight triangulation $X$. 
By Theorem \ref{theo:main}, $X$ is stacked. But, by Corollary 3.13  (case $d=3$) of \cite{DM}, any stacked triangulation of a closed 3-manifold can be obtained from a stacked 3-sphere by a finite sequence of elementary handle additions.   
It is easy to see by an induction on the number $k$ of handles added that $X$ triangulates either $S^3$ ($k=0$) or $(S^2\times S^1)^{\# k}$ or 
$(\TPSS)^{\# k}$ ($k\geq 1$). Let $X$ be obtained from the stacked 3-sphere $S$ by  $k$ elementary handle additions. It follows by induction on $k$ that $f_0(S) = f_0(X)+4k$ and $f_1(S) = f_1(X)+6k = \binom{f_0(X)}{2}+6k$. Since $S$ is a stacked 3-sphere, $f_1(S) = 4f_0(S) -10$. Thus, $\binom{f_0(X)}{2}+6k = 4(f_0(X) +4k) -10$. This implies $(f_0(X)-4)(f_0(X) -5) =20k$ and hence $f_0(X) = \frac{1}{2}(9+\sqrt{80k+1})$. So, $80k+1$ must be a perfect square. 
\end{proof} 

\begin{proof}[Proof of Corollary $\ref{coro:Ku}$]
If $(f_0(M)-4)(f_0(M)-5) = 20 \beta_1(M; \FF)$ then Theorem 1.3 of \cite{BDSS1} says that $M$ must be neighbourly and locally stacked. Therefore, the `if part' follows from  Theorem 2.24 of \cite{BDStacked}. The `if part' also follows from Theorem 1.12 of \cite{BaMu} and Theorem 2.24 of \cite{BDStacked}. The `only if' part follows from the proof of Corollary \ref{coro:S2S1} above.
\end{proof} 

\medskip

\noindent {\bf Acknowledgements:} 
The second and third authors are supported by  DIICCSRTE, Australia (project
AISRF06660) and DST, India (DST/INT/AUS/P-56/2013(G)) under the Australia-India Strategic Research Fund. The
second author is also supported by the UGC Centre for Advanced Studies (F.510/6/CAS/2011 (SAP-I)). The authors thank Wolfgang K\"{u}hnel for some useful comments on the earlier version of this paper. 

{\small

}

\end{document}